\documentclass[a4paper,12pt,oneside,reqno]{amsart}
\usepackage{amssymb,amsfonts,amsmath,amsthm,amscd, bbm}
\usepackage{graphicx, color}
\usepackage{bm}
\usepackage{comment}
\usepackage{cases}

\numberwithin{equation}{section}  

\newtheorem{teor}{Theorem}[section]
\newtheorem{prop}[teor]{Proposition} 
\newtheorem{lem}[teor]{Lemma}

\newcommand{\ugdef}{\overset{\text{\text{def}}}{=}}

\newcommand{\xj}{X^{J}_{\alpha_{\tiny{J}},i}}
\newcommand{\gt}{G^T}
\newcommand{\pgt}{\left(G^T\right)}
\newcommand{\alj}{\alpha_{\tiny{J}}}

\newcommand{\Bk}{ T_k}
\newcommand{\Bj}{T_j}
\newcommand{\Bkj}{ {T}_{k_J}}

\newcommand{\chset}{\mathcal{C}_n}
\newcommand{\relen}[2]{ H(#1\,|\,#2) }

\DeclareMathOperator{\att}{\mathbb{E}}
\DeclareMathOperator{\prob}{\mathbb{P}}
\DeclareMathOperator{\var}{\mathbb{V}\text{ar}}
\newcommand{\prosp}{ \mathcal{M}^+_1(S)}
\newcommand{\totprosp}{\mathcal{M}^+_1(S^{2^n-1})}
\newcommand{\fprosp}[1]{\mathcal{M}^+_1 \left(S^{#1}\right) }
\newcommand{\Pset}[1]{\mathcal{P}_{#1}}
\newcommand{\vx}[1]{\boldsymbol{x}^{(#1)}}

\newcommand{\beq}{\begin{equation}}
\newcommand{\eeq}{\end{equation}}
\newcommand{\bea}{\begin{aligned}}
\newcommand{\eea}{\end{aligned}}

\newcommand{\R}{\mathbb R}

\begin{document}

\title[On a nonhierarchical Perceptron GREM]{On a nonhierarchical generalization \\ of the Perceptron GREM}

\author[N. Kistler]{Nicola Kistler}
\address{Nicola Kistler \\ J.W. Goethe-Universit\"at Frankfurt, Germany.}
\email{kistler@math.uni-frankfurt.de}

\author[G. Sebastiani]{Giulia Sebastiani}
\address{Giulia Sebastiani \\ J.W. Goethe-Universit\"at Frankfurt, Germany.}
\email{sebastia@math.uni-frankfurt.de}

\thanks{This work has been supported by a DFG research grant, contract number 2337/1-1.}


\date{\today}

\begin{abstract}
We introduce a nonlinear, nonhierarchical generalization of Derrida's GREM and establish through a Sanov-type large deviation analysis both a Boltzmann-Gibbs principle as well as a Parisi formula for the limiting free energy. In line with the predictions of the Parisi theory, the free energy is given by the minimal value over all Parisi functionals/hierarchical structures in which the original model can be coarse-grained. 
\end{abstract}

\maketitle

\section{Introduction} 

The study of mean field spin glasses, started by physicists in the 1970's and culminated in the Parisi theory \cite{mpv}, has over time revealed a rich and challenging mathematical structure. Despite the enormous progress over the last decades, the whole Parisi picture and its presumed universality for mean field models still remains not fully understood. Indeed, for the prototypical models of the theory such as the  Sherrington-Kirkpatrick [SK] model \cite{sk}, a rigorous treatment via large deviations techniques within the classical Gibbs-Boltzmann formalism is still lacking, and one has to resort to sophisticated machinery, in particular Guerra's interpolations \cite{g, t} and the Ghirlanda-Guerra identities \cite{gg}. These tools are remarkably efficient but cannot provide, by their own nature, insights for given realization of the disorder. We refer the reader to  \cite{panchenko} and references therein for a comprehensive account of the state of the art for the SK-model\footnote{A treatment of spin glasses which is somewhat complementary to the Parisi theory goes through the TAP-Plefka framework \cite{tap, plefka}. The rigorous studies which bypass the interpolations/GG-identities employ the Kac-Rice formula and some ensuing large deviations estimates for the complexity/number of critical points of Gaussian random fields. This type of analysis is to date however limited to spherical models with REM-like ultrametric structure, see \cite{subag_one, subag_two} and references.}. 

One of the most puzzling features of the Parisi theory is the {\it ultrametricity}, namely the emergence in the infinite volume limit of hierarchical structures closely related to the generalized random energy models, the GREM, introduced by Derrida in the 1980's \cite{d_rem, d}. For a large class of SK-type Hamiltonians, the ultrametricity has meanwhile been put on rigorous ground by Panchenko \cite{p_ultra}. The analysis  relies crucially on the Ghirlanda-Guerra identities. 

In order to shed some light on the {\it origin} of ultrametricity, simplified models for spin glasses which are amenable to a rigorous solution via large deviation techniques have been introduced by Bolthausen and the first author in \cite{bk, bk1, bk2}: the first paper deals with a nonhierarchical yet linear superposition of Derrida's GREMs, whereas the second paper analyses a non-linear version of the GREM which is closely related to the perceptron model from neural networks \cite{mpv} and which we refer to henceforth as {\it Perceptron GREM}.  

In the present paper we introduce a random Hamiltonian which is neither hierarchical from the start nor given by linear superpositions of GREM-like models. For such nonhierarchical, nonlinear generalization of the random energy models we establish a quenched, Sanov-type large deviation principle for the empirical distribution of the random energies which in turn allows to derive a Boltzmann-Gibbs principle, as well as a Parisi formula for the free energy in the thermodynamical limit. The limiting free energy turns out to be given by the minimal value over {\it multiple}  nonlinear Parisi functionals / hierarchical structures.

What is perhaps more, our analysis suggests that the onset of ultrametricity is intimately related to certain monotonicity properties of relative entropies/Kullback-Leibler functionals, and shows how these monotonicities lead to the optimal balance in the energy-entropy competition. 

We emphasize that the spin glass we investigate here is {\it nonlinear} insofar the sub-energies interact via a perceptron-like function, see \eqref{original_Ham} below for the definition of the model, but the Hamiltonian still is a linear functional of the empirical measures of the  random energies. The case of nonlinear functionals, which give rise to models which better approximate realistic spin glasses such as the SK-model, will be addressed in a forthcoming work.

\section{Definition of the model and main results}

We start with some notation: for $n\in\mathbb{N}$, we set $I\ugdef \{1,2,...,n\}$ for the set of {\it species}. 
For $J\subseteq I$ we denote by  $$\Pset{J}\ugdef\{A\subseteq J: A\neq\emptyset\}$$
the power set of $J$ (without empty set).  For $N\in\mathbb{N}$, which will play the role of the volume, and $j \in I$,  we denote by $$\alpha_j\in\left\{1,\cdots,2^{\frac{N}{n}}\right\}$$ the {\it spins} associated to the $j^{th}$-species. For $J\in\Pset{I}$ we identify the $|J|$-tuple $(\alpha_j)_{j\in J}$ with 
$$\alj\ugdef(\alpha_j)_{j\in J},\qquad\alj\in\left\{1,\cdots,2^{\frac{|J|N}{n}}\right\}.$$

By a slight abuse of notation, if $J= I$ we shorten $\Pset{n}$ for $\Pset{I}$, and write 
$\alpha$ for a {\it configuration}, i.e. the whole vector of spins $(\alpha_1,\cdots,\alpha_n)$. 

Remark that we have altogether $2^N$ configurations, and since our configuration space is in one-to-one correspondence with the set $\{1, \dots, 2^N\}$, to lighten notations, we will refer to the latter in the counting procedures which we implement below.  

Consider independent random variables $$\left\{\xj \right\}_{J\in\Pset{n}, i \in \mathbb N}$$ with values in a given Polish space $S$ equipped with the Borel $\sigma$-field $\mathcal{S}$, and defined on a probability space $(\Omega,\mathcal{A},\mathbb{P})$. We assume that for $J\in\Pset{n}$ each variable $\xj$ is distributed according to the same probability distribution $\mu_J$. In this case, the product measure $\mu \ugdef \bigotimes\limits_{J\in\Pset{n}}\mu_J$ gives the joint distribution of the random vector 
\beq \label{original_vec}
X_{\alpha,i}\ugdef\left( \xj \right)_{ J\in\Pset{n} }.
\eeq

Since $|\Pset{n}|=2^n-1$, the joint law $\mu$ is an element of  $\mathcal{M}^+_1(S^{2^n-1})$, the space of all probability distributions on the product space $S^{2^n-1}$ which, endowed with the topology of weak convergence, is Polish itself. 

For $d\in\mathbb N$, $\nu\in \mathcal{M}^+_1(S^{d})$, we denote by $B_r(\nu)$ the open ball centered in $\nu$ with radius $r>0$ in one of the standard metrics (e.g. Prokhorov's metric); a {\it{neighborhood}} or an {\it{open subset}} are conceived in the corresponding topology. This turns $\mathcal{M}^+_1(S^{d})$ into a complete, separable metric space. 

For $d\in\mathbb N$, $\nu\in\mathcal M_1^+(S^d)$ and $B$ a subset of indeces, $B\subseteq\{1,\cdots,d\}$, we write $\nu^{(B)}\in\fprosp{|B|}$ for the marginal of $\nu$ restricted to the coordinates corresponding to indices in $B$;  with $\pi_{B}: S^{d} \to S^{|B|}$ 
the associated natural projection, it thus holds that 
$$\nu^{(B)}\ugdef \left( \pi_{B} \right)_*(\nu)=\nu \circ \pi_{B}^{-1}.$$
Given a function $\phi: S^{2^n-1}\to\mathbb{R}$, our spin glass Hamiltonian is then 
\begin{equation}\label{original_Ham}
H_N(\alpha) = \sum\limits_{i\leq N} \phi\left(X_{\alpha,i} \right) = N \int \phi(x) L_{N, \alpha}(dx)\,,
\end{equation}
where 
\begin{equation}\label{empirical_distr} 
L_{N,\alpha} = \frac{1}{N}\sum\limits_{i=1}^N \delta_{X_{\alpha,i}}\,,
\end{equation}
is the empirical measure, a random element in $\mathcal M_1^{+}\left(S^{2^n-1}\right)$. The partition function and free energy associated  to \eqref{original_Ham} are defined as
\begin{equation} \label{partition_function}
Z_{N}(\phi) = 2^{-N}\sum\limits_{\alpha=1}^{2^N} e^{H_N(\alpha)}, \qquad \text{resp.} \quad F_{N}(\phi)=\frac{1}{N}\log Z_N(\phi). 
\end{equation}
This spin glass model is thus a generalization of the Perceptron GREM discussed in \cite{bk2} but which, unlike the latter, has no in-built ultrametric structure.  

For $d\in\mathbb N$, $\nu, \nu' \in \mathcal M_1^+(S^d)$, we denote by 
$$H(\,\nu\,|\,\nu'\,) \ugdef \begin{cases}
\int \log \frac{d \nu}{d \nu'} d\nu & \text{if}\quad \nu \ll \nu' \\
\infty & \text{otherwise.} 
\end{cases}$$
the usual relative entropy of $\nu$ with respect to $\nu'$. Finally, for $J\in \Pset{n}$, we set
\begin{equation} \label{def_cond}
C_J \ugdef \left\{ \nu \in \totprosp : \relen{ \nu^{(\Pset{J})} }{ \mu^{(\Pset{J})} } \leq \frac{|J|}{n} \log2 \right\}.
\end{equation}

\noindent Our first result states that the quenched infinite volume free energy exists, and is given by the solution of a Boltzmann-Gibbs variational principle. The upshot is analogous to the control of the thermodynamical limit  established in \cite{bk} through second moment estimates on the energy levels for the case of linear $\phi$, and with disorder given by Gaussian random variables.

\begin{teor} \label{teor_Gibbs}
Let $\phi: S^{2^n-1}\to\mathbb{R}$ be continuous and bounded. 
Then $f(\phi)\ugdef\lim\limits_{N\to\infty} F_N(\phi)$ exists almost surely, and is given by
\begin{equation} \label{quenched_limit}
f(\phi)= \sup_{\nu\in\totprosp} \left\{ \int \phi d\nu - \mathfrak J(\nu) \right\} 
\end{equation}
where the \emph{quenched rate function} $\mathfrak J:\totprosp\to \overline{\mathbb{R}}$  is given by
\begin{equation} \label{def_J}
\mathfrak J(\nu)\ugdef\begin{cases}
H(\nu|\mu) & \mbox{if} \quad \nu \in \bigcap\limits_{J\in \Pset{n}} C_J \\
\infty & \mbox{otherwise}. 
\end{cases}
\end{equation}

\end{teor}

Since the set $\Pset{n}$ is partially ordered through the inclusion relation, we can consider its totally ordered subsets, which are referred to, with the terminology from \cite{bk}, as  {\it chains}. We set

\begin{equation} \label{def_chains}
\chset \ugdef \left\{ T = \left\{ A^T_k \right\}_{k=0}^n : A^T_0=\emptyset, \, \forall k \in I \setminus \{n\}:\; A^T_k\subset I, \, A^T_k \subset A^T_{k+1}, \, |A^T_k|=k, \, A^T_n = I \right\}.
\end{equation}

Each element in $\chset$ corresponds to a specific hierarchical structure among all those in which the original model can be {\it coarse-grained}, i.e. to each chain $T\in\chset$ we can associate a Hamiltonian $H^T_N$ which gives rise to a Perceptron GREM.  To see how this goes, for $T\in\chset$, we define 
\begin{equation} \label{B_k}
\Bk\ugdef \{ J\subset I: J\subseteq A^T_k, J\not\subseteq A^T_{k-1}\} \subset \Pset{n}\,,\quad k\in I.
\end{equation}
The sets ${T}_1,\cdots, {T}_n$ describe explicitely how a specific chain $T$ assembles the $n$-levels into the corresponding hierarchical structure: if
\beq \label{chain_explicit}
T = \left\{ \,\emptyset \subset \{a_1\} \subset  \{a_1, a_2\} \subset \cdots \subset \{a_1, \cdots, a_{n-1}\} \subset \{1,\cdots,n\} \, \right\}
\eeq
then the $\alpha_{a_1}$'s play the role of indeces for the first level, $(\alpha_{a_1},\alpha_{a_2})$'s for the second etc..\\
Since $\Pset{n}$ is the disjoint union of the $\{\Bk\}_{k\in I}$ for any $J\in\Pset{n}$ there is a unique
\beq \label{index_k}
k_{J} = k_J(T) \in I \; \text{for which} \;  J \in\Bkj \,.
\eeq
Given these indeces we shorten\footnote{We stress that the order in which the random sub-energies appear in the vectors \eqref{vec_T} is the same as the one of the original vector \eqref{original_vec}: if $\Pset{n}=\{J_1, \cdots, J_{2^n-1}\}$, then the $c$-th variable is distributed according to $\mu_{J_c}$, $\forall \,\, c\in\{1,\cdots,2^{n}-1\}$.}
\beq \label{vec_T}
Y^T_{\alpha,i}=\left(Y^J_{\alpha_{a_1},\cdots\alpha_{a_{k_J}},i}\right)_{J\in\Pset{n}},
\eeq
where the $Y^J_{\cdot,\cdot}$-s are $\mu_J$-distributed and independent both over the $\alpha$'s as well as over the $i$'s. Hence, the coarse grained Hamiltonian corresponding to the chain $T\in\chset$ is defined as
\begin{equation}\label{def_Ham_T}
H^T_N(\alpha)\ugdef \sum\limits_{i=1}^N \phi(Y^T_{\alpha,i})\,,
\end{equation}
with associated partition function and free energy in finite volume denoted by
\begin{equation}\label{def_Z_F_T}
Z^T_N(\phi) =2^{-N} \sum\limits_{\alpha=1}^{2^N}e^{H^T_N(\alpha)},\quad F^T_N(\phi) = \frac{1}{N}\log Z^T_N(\phi).
\end{equation}
The Hamiltonian \eqref{def_Ham_T} is in the form of a Perceptron GREM\footnote{   
The coarse-graining requires some burdening notation, but the procedure is quickly understood through a concrete example which the reader may want to keep in mind throughout: let $n=2$, and consider $X_{\alpha,i}=( X^{\{1\}}_{\alpha_1,i},\,\,\,X^{\{2\}}_{\alpha_2,i},\,\,\,X^{\{1,2\}}_{(\alpha_1,\alpha_2),i})\sim \mu_{\{1\}}\otimes\mu_{\{2\}}\otimes\mu_{\{1,2\}}$, in which case 
$ \phi: S^3\to \R$ and $H_N(\alpha)=\sum_i \phi(X^{\{1\}}_{\alpha_1,i},\,\,\,X^{\{2\}}_{\alpha_2,i},\,\,\,X^{\{1,2\}}_{\alpha_1,\alpha_2,i})$.  There are two possible chains: $R\ugdef \left\{ \{1\},\,\{1,2\} \right\}$ or $T\ugdef \left\{ \{2\},\,\{1,2\} \right\}$. By way of example we focus on the latter, in which case the sets \eqref{B_k} are given by ${T}_1= \left\{ \{2\} \right\} , {T}_2= \left\{ \{1\}, \{1,2\} \right\}$, and the new sub-energies are given by
$Y^{\{2\}}_{\alpha_2,i}\sim \mu_{\{2\}}$ for the "first level", and $(Y^{\{1\}}_{(\alpha_1, \alpha_2),i}, Y^{\{1,2\}}_{(\alpha_1,\alpha_2),i})\sim \mu_{\{1\}}\otimes \mu_{\{1,2\}}$ for the second. The Hamiltonian of the coarse-grained model is then $H_N^T(\alpha) = \sum_i \phi( Y_{\alpha_1, \alpha_2, i}^{\left\{1\right\}}, Y_{\alpha_2, i}^{\left\{2\right\}},  Y_{\alpha_1, \alpha_2, i}^{\left\{1, 2\right\}} )$. The key aspect of the procedure is thus to replace the $X^{\{1\}}_{\alpha_1,i}$ with $Y^{\{1\}}_{\alpha_1, \alpha_2,i}$ the latter being independent over the $\alpha_2's$: this eventually induces a hierarchical correlation structure as in the Perceptron GREM.}
introduced and solved in \cite{bk2}:  both a  Gibbs variational principle, as well as a Parisi-like formula, for the limiting free energy are available.

For later purposes, we need to recall how the Parisi variational principle for the Perceptron GREM is constructed \cite{bk2}: consider the standard simplex 
\begin{equation}
\Delta \ugdef \left\{ \boldsymbol m = (m_1, \dots, m_n): \; 0 < m_1 \leq m_2 \leq \dots \leq m_n \leq 1 \right\}\,, 
\end{equation}
and recursively construct functions 
\begin{equation} \begin{aligned} \label{def_phik}
\phi^T_n &= \phi, \\
\phi^T_{k-1}\left(\boldsymbol{x}^{(1)},\cdots,\boldsymbol{x}^{(k-1)}\right) & = \frac{1}{m_k} \log \int \,\, \exp \left[ m_k\phi^T_k\left(\boldsymbol{x}^{(1)},\cdots, \boldsymbol{x}^{(k-1)}, \boldsymbol{y} \right) \right] \mu^{(\Bk)}(d\boldsymbol{y})
\,,
\end{aligned} \end{equation}
for $k\in I$, shortening also $\boldsymbol{x}^{(j)}=(x_J)_{J\in T_j}$ and $ \mu^{(\Bk)}(d\boldsymbol{y})=\prod_{J \in \mathcal P_k^T \setminus \mathcal P_{k-1}^T}\mu_J(d y_J)$, and omitting the dependence of the $\phi$-functions on the parameter $\boldsymbol m$ to lighten notations.

We then set
\begin{equation}\label{func_Par}
P^T(\boldsymbol m)  \ugdef \frac{\log2}{n}\sum\limits_{k=1}^n \frac{1}{m_k}+\phi_0^T(\boldsymbol m)\,.
\end{equation}

It is proved in \cite[Theorem 2.3-2.5]{bk2} that for continuous and bounded $\phi$, the limit $f^T(\phi)\ugdef\lim\limits_{N\to\infty}F^T_N$ exists almost surely and is given by the solution of the Boltzmann-Gibbs variational principle
\begin{equation}\label{free_en_T}
f^T(\phi) = \sup_{\nu\in\totprosp} \left\{ \int \phi(x) d\nu - \mathfrak J^T(\nu) \right\}\,,
\end{equation}
with quenched rate function $\mathfrak J^T:\totprosp\to \overline{\mathbb{R}}$ given by
\begin{equation}\label{def_J_T}
\mathfrak J^T(\nu)\ugdef\begin{cases}
H(\nu|\mu) & \mbox{if} \quad \nu \in \bigcap\limits_{k=1}^n C_{A^T_k} \\
\infty & \mbox{otherwise}, \end{cases}
\end{equation}
and the sets $C_{A^T_k}$ as in \eqref{def_cond}. Moreover, the solution can be given in terms of the following Parisi variational principle
\begin{equation}\label{Parisi_T}
f^T(\phi) = \min_{\boldsymbol{m}\in\Delta} P^T(\boldsymbol{m})\,.
\end{equation}

\noindent Our contribution here is to extend the above to the {\it nonhierarchical} Hamiltonian \eqref{original_Ham}. 

\begin{teor}\label{solution_T_Par} Assume $\phi: S^{2^n-1} \to \mathbb{R}$ is continuous and bounded. Then
\begin{equation}\label{free_en_Par}
f(\phi) = \min\limits_{T\in\chset} f^T(\phi) =  \min\limits_{T\in\chset} \left\{ \min_{\boldsymbol{m}\in\Delta} P^T(\boldsymbol{m}) \right\}
\end{equation}
where $f(\phi)$ is the limiting free energy  \eqref{quenched_limit}, and $f^T(\phi)$ the limiting free energy \eqref{free_en_T} associated to the chain $T\in\chset$.
\end{teor}
In other words, the limiting free energy of the nonhierarchical Perceptron GREM is given by the minimum among the free energies associated to all chains in which the original model can be coarse-grained. The novelty compared to \cite{bk2} is thus the minimization in \eqref{free_en_Par} over all $T \in \mathcal C_n$, i.e. over all possible hierarchical structures: this feature is in full agreement with the Parisi theory for the more sophisticated mean field spin glasses like the SK-model, where the corresponding minimization is indeed over {\it all $q$-functions} \cite{mpv}.

\vskip0.5cm

\section{The Gibbs variational principle}

In this section we give a proof of Theorem \ref{teor_Gibbs}. For this we will closely follow the large deviation analysis of \cite{bk2} of the Perceptron GREM, which is flexible, and carries over to this more general setting quite swiftly. Lest the reader is overwhelmed with innumerous pointers and references, and given that the treatment is reasonably short, we will work out out all steps in detail. 

We first need some infrastructure: for an open subset $U\subset \totprosp$ we will denote by $M_N(U)$ the random variable counting the number of $\alpha$-s for which the corresponding empirical distribution \eqref{empirical_distr} lies in $U$, to wit:
$$M_N(U) = \#\{ \alpha=1,\dots,2^N : L_{N,\alpha} \in U \}.$$

The following two lemmata state some well known properties of relative entropies.  The short and elementary proofs are given for completeness.

\begin{lem} \label{relen_basic_prop}
Let $S$ be a Polish space, $d\in\mathbb{N}$, $\nu,\mu \in \fprosp{d}$ with $\nu\ll\mu$, $\mu=\bigotimes_{k=1}^d \mu_k$ for some $\mu_k\in\prosp$, $k\in\{1,\cdots,d\}$. Then for every nonempty $B\subseteq \{1,\cdots,d\}$, \begin{itemize}
\item[] 
\begin{equation} \label{chain_rule}
\relen{\nu}{\mu} = H\left( \, \nu^{(B)} \, |  \,  \mu^{(B)}   \, \right) + H\left( \,  \nu   \,   |  \, \nu^{(B)}\otimes \mu^{(B^c)}    \, \right)
\end{equation}

\item[] \begin{equation}\label{AM}
H\left(  \nu^{(B)} \,|\,  \mu^{(B)}   \right) \leq H\left( \nu \, | \, \mu \right) 
\end{equation}


\end{itemize}
where $B^c = \{1,\cdots,d\}\setminus B$.
\end{lem}

\begin{proof}
The second claim \eqref{AM} is an immediate consequence of \eqref{chain_rule}, together with the positivity of relative entropies. In order to see the first claim, let $\{S_k\}_{k=1}^d$ be $d$ copies of $S$ and $B=\{b_1,\cdots,b_M\}$, $B^c=\{a_1,\cdots,a_{d-M}\}$. 
Since the product space of Polish spaces is Polish, specifically Radon,the disintegration Theorem holds for every $\nu\in\fprosp{d}$ and we can define the conditional distribution $\nu^{(B^c|B)}$ of $\nu$ on $S_{a_1} \times \cdots \times S_{a_{d-M}}$ given the projection on $S_{b_1}\times \cdots \times S_{b_M}$. Expressing $\nu = \nu^{(B)} \otimes \nu^{(B^c \mid B)}$ as a semi-direct product, it holds 
\beq \bea \label{comi}
\relen{\nu}{\mu}  = \int \log\left(\frac{d\nu}{d\mu}\right) d\nu &= \int  \log\left(\frac{d\nu^{(B)}}{d\mu^{(B)}} \right) d\nu + \int \log \left( \frac{d\nu^{(B^c|B)}}{d\mu^{(B^c)}} \right) d\nu \\
& = H\left( \, \nu^{(B)} \, |  \,  \mu^{(B)}   \, \right)+  \int \log \left( \frac{d\nu^{(B^c|B)}}{d\mu^{(B^c)}} \right) d\nu \,.
\eea \eeq
We now focus on the second term on the r.h.s. above: by well-known properties of Radon-Nikodym derivatives\footnote{This step is somewhat informal, but see e.g. \cite{polyanski}.} we may write
\beq \bea \label{comi_uno}
\frac{d\nu^{(B^c|B)}}{d\mu^{(B^c)}} & = \frac{d\nu^{(B)}}{d\nu^{(B)}} \cdot \frac{d\nu^{(B^c|B)}}{d\mu^{(B^c)}} = \frac{d \nu}{d \left( \nu^{(B)} \otimes \mu^{(B^c)}\right)} \,,
\eea \eeq
and therefore 
\beq \label{comi_due}
 \int \log \left( \frac{d\nu^{(B^c|B)}}{d\mu^{(B^c)}} \right) d\nu = \int \log\left( \frac{d \nu}{d \left( \nu^{(B)} \otimes \mu^{(B^c)}\right)} \right) d\nu = H\left( \,  \nu   \,   |  \,\nu^{(B)}\otimes \mu^{(B^c)}    \, \right).
\eeq
Plugging this in \eqref{comi} then yields \eqref{chain_rule}.

\end{proof}

\begin{lem}(Semicontinuity of relative entropies) \label{semicont}
With the notation of the previous Lemma, the functions $\nu \to \relen{\nu}{\mu}$, $\nu \to H\left( \,  \nu   \,   |  \,   \nu^{(B)}\otimes \mu^{(B^c)}  \, \right)$ are lower semicontinuous in the weak topology.
\end{lem}

\begin{proof}
It is well known that for relative entropies the following representations hold:
\begin{equation} \label{repr_rel_en}
\relen{\nu}{\mu} = \sup\limits_{u\in \mathcal{U}} \left[ \int u d\nu -\log \int e^ud\mu\right] ,
\end{equation}
where $\mathcal{U}$ is the set of all bounded continous functions $S^d\to \mathbb{R}$, and analogously for
\beq \label{repr_rel_ent}
H\left( \,  \nu   \,   |  \,   \nu^{(B)}\otimes \mu^{(B^c)}  \, \right) = \sup\limits_{u\in \mathcal{U}} \left[ \int u d\nu -\log \int e^ud\left( \nu^{(B)}\otimes \mu^{(B^c)} \right)\right].
\end{equation}
But for $u\in\mathcal{U}$, the functionals $$\nu\to\int u d\nu, \quad \nu\to\log \int e^u d\mu,\quad \nu\to\log \int e^ud\left( \nu^{(B)}\otimes \mu^{(B^c)} \right)$$ 
are all continuous: the claim of the Lemma thus steadily follows from this observation and the representations
\eqref{repr_rel_en} and \eqref{repr_rel_ent}.
\end{proof}

The next technical result is a generalization of \cite[Lemma 3.1]{bk2}.

\begin{lem}(Relative Sanov Theorem) \label{relative_Sanov}
Let $\alpha=(\alpha_1,\cdots,\alpha_n)$, $\alpha'=(\alpha'_1,\cdots,\alpha'_n)$ be two configurations and $L_{N,\alpha}, L_{N,\alpha'}$ the associated empirical distributions defined in \eqref{empirical_distr}. Assume that for $A\in \Pset{n},$ 
$$\alpha_j=\alpha'_j,\,\; j\in A\; \qquad \text{and} \qquad\alpha_j\neq\alpha'_j, \; j\in I\setminus A.$$
Then the couple $(L_{N,\alpha}, L_{N,\alpha'})$ satisfies a LDP with good rate function 
\begin{equation}
\mathfrak J_A(\nu,\theta)= \relen{ \nu^{ \left( \Pset{A} \right) } }{ \mu^{ \left( \Pset{A} \right)} }+  H\left( \,  \nu   \,   |  \,  \nu^{\left( \Pset{A}\right)}\otimes \mu^{\left(\Pset{A}^c \right)}   \, \right)+ H\left( \,  \theta   \,   |  \,  \theta^{\left( \Pset{A}\right)}\otimes \mu^{\left(\Pset{A}^c \right)}   \, \right) ,
\end{equation}
if $\nu^{ \left( \Pset{A} \right)} = \theta^{ \left( \Pset{A} \right)}$ (and $=\infty$ otherwise), where $\mathcal P_A^c=\Pset{n}\setminus \Pset{A}$.
\end{lem}
\begin{proof}
Let $S$ be a Polish space, $d\in\mathbb{N}$, 
$$\mathcal{P}'=\{b_1,\cdots,b_M\}\subseteq \{1,\cdots,d\}, \mathcal{P}'\neq\emptyset,$$
 and $\{\boldsymbol{X}_i\}$, $\{\boldsymbol{Y}_i\}$ two independent families of indendent random variables with 
$$\boldsymbol{X}_i \sim \mu \in \bigotimes\limits_{k=1}^d\mathcal{M}^+_1(S_k),
\qquad \boldsymbol{Y}_i \sim \mu' \in \bigotimes\limits_{l=1}^{M}\mathcal{M}^+_1(S'_{b_l}),$$
where each $S_k,S'_{b_l}$ is a copy of $S$.\\We consider the empirical distribution
\begin{equation} \label{M_N_emp_distr}
M_N = \frac{1}{N}\sum_{i=1}^N \delta_{(\boldsymbol{X}_i,\boldsymbol{Y}_i)}
\end{equation}
associated to vectors $(\boldsymbol{X}_i, \boldsymbol{Y}_i)\in S_1\times\cdots\times S_d \times S'_{b_1}\times\cdots\times S'_{b_M}$
and the natural projections 
\beq \label{projs}
\bea
  \pi & : S_1\times\cdots\times S_d \times S'_{b_1}\times\cdots\times S'_{b_M} \longrightarrow S_1\times \cdots \times S_d,\\
 \pi' &: S_1\times\cdots\times S_d \times S'_{b_1}\times\cdots\times S'_{b_M} \longrightarrow \tilde{S}_1\times \cdots  \times \tilde{S}_d,
\eea 
\eeq 
where
$$ \tilde{S}_k \ugdef
\begin{cases}
S_k\quad& \text{if}\quad k \in \mathcal{P} \\
S'_k\quad& \text{if}\quad k\notin \mathcal{P} \\
\end{cases}\quad,\qquad \Pset{} \ugdef \{1,\cdots,d\} \setminus \mathcal{P}' = \{ c_1, c_2, \cdots, c_{d-M} \}.$$
We define 
$$ (L_N, R_N) \ugdef M_N\circ (\pi,\pi')^{-1},$$
\vskip0.2cm
which is a random element in $\mathcal M_1^+(S_1\times\cdots\times S_d)\times \mathcal M_1^+(\tilde{S}_1\times\cdots\times \tilde{S}_d)$, and 
\beq 
\mathfrak J(\nu,\theta) \ugdef \begin{cases}
\relen{ \nu^{ \left( \Pset{} \right) } }{ \mu^{ \left( \Pset{} \right)} }+  H\left( \,  \nu   \,   |  \,  \nu^{\left( \Pset{}\right)}\otimes \mu^{\left(\Pset{}' \right)}   \, \right)+ H\left( \,  \theta    \,   |  \,  \theta^{\left( \Pset{}\right)}\otimes \mu'  \, \right) & \mbox{if} \quad \nu^{ \left( \Pset{} \right)}=\theta^{ \left( \Pset{} \right)} \\
\infty & \mbox{otherwise.}\end{cases}
\eeq
Applying Sanov's Theorem to \eqref{M_N_emp_distr}, and by continuous projection, we obtain that the couple $(L_N,R_N)$ satisfies a LDP with good rate function
$$\mathfrak J'(\nu,\theta) = \inf\limits_{\rho}\left\{ H\left(\, \rho \, | \, \mu\otimes\mu' \,\right) : \quad \rho\pi^{-1} = \nu, \quad \rho\pi'^{-1} = \theta \right\}.$$
Since the projections \eqref{projs} coincide on coordinates indexed in $\mathcal{P}$, for all \\
$\rho\in\mathcal{M}^+_1\left( S_1\times\cdots\times S_d \times S'_{b_1}\times\cdots\times S'_{b_M} \right)$ it holds $(\rho\pi^{-1} )^{(\Pset{})} = (\rho\pi'^{-1})^{(\Pset{})}$ and the following implication is immediate 
\[
\nu^{(\Pset{})}\neq \theta^{(\Pset{})} \Longrightarrow \mathfrak J'(\nu,\theta)=\infty.
\] 
 We may thus henceforth assume $\nu^{(\Pset{})} = \theta^{(\Pset{})}$, and consider $\tilde{\rho}=\tilde{\rho}(\nu,\theta)$ with marginal $\nu^{(\Pset{})}=\theta^{(\Pset{})}$ on $S_{c_1}\times\cdots \times S_{c_{d-M}}$, and conditional distribution on $S_{b_1}\times\cdots \times S_{b_{M}} \times S'_{b_{1}}  \times \cdots \times S'_{b_{M}}$ given the projection on $S_{c_1}\times\cdots \times S_{c_{d-M}}$  as the product of the conditional distribution of $\nu$ on $S_{b_1}\times\cdots\times S_{b_M} $ given the projection on $S_{c_1}\times\cdots \times S_{c_{d-M}}$, and of the conditional distribution $\theta$ on $S'_{b_1}\times\cdots\times S'_{b_M} $ given the projection on $S_{c_1}\times\cdots \times S_{c_{d-M}}$. In which case, applying \eqref{chain_rule}, we easily obtain
$$H\left( \,  \tilde{\rho}   \,   |  \,  \mu\otimes \mu'  \, \right)
=H\left( \,  \nu^{(\Pset{})}   \,   |  \,  \mu^{(\Pset{})}   \, \right) + H\left( \,  \nu   \,   |  \,   \nu^{(\Pset{})}\otimes \mu^{(\Pset{}')}  \, \right)+H\left( \,  \theta   \,   |  \,   \theta^{(\Pset{})}\otimes \mu'  \, \right)$$
and therefore $\mathfrak J\geq \mathfrak J'$. Conversely, for any $\rho$ satisfying $\rho\pi^{-1}=\nu$, $\rho\pi'^{-1}=\theta$, we claim that 
\begin{equation} \label{claimo}
\mathfrak J(\nu,\theta)\leq \relen{\rho}{\mu\otimes\mu'}\,.
\end{equation} 
(Remark that we can assume that the right hand side is finite). To see this, we first write
$$H\left( \,  \rho   \,   |  \,  \mu\otimes\mu'   \, \right)=H\left( \,  \rho   \,   |  \,  \tilde{\rho}(\nu,\theta)   \, \right)+\int d\rho \log \frac{d \tilde{\rho}(\nu,\theta)}{d( \mu\otimes\mu')}.$$
The first term is non-negative, while the second equals
$$\int d\tilde{\rho}(\nu,\theta) \log \frac{d \tilde{\rho}(\nu,\theta)}{d( \mu\otimes\mu')} = \mathfrak J(\nu,\theta),$$ and \eqref{claimo} follows. The statement of the Lemma corresponds then to the case 
\[
d=2^n-1, \qquad \Pset{}'=\mathcal P_A^c.
\]

\end{proof}

\begin{prop} \label{prop_1_G}
Let $\nu\in\fprosp{2^n-1}$ be such that $\relen{\nu}{\mu} < \infty$, $\epsilon >0$ and $V$ an open neighborhood of $\nu$. Then there exists an open neighborhood $U$ of $\nu$, $U\subset V$, and $\delta >0$ for which
$$ \prob\left[ M_N(U) \geq e^{N(\log2-\relen{\nu}{\mu}+\epsilon)}\right]\leq e^{-\delta N}.$$
\end{prop}

\begin{proof}
By the semicontinuity property of the relative entropy, given $B_r(\nu)$ a family of open balls with radius $r>0$, one has 
$$\inf\limits_{\rho\in B_r(\nu)} \relen{\rho}{\mu} \to \relen{\nu}{\mu}\qquad\text{as}\quad r\to 0.$$
Specifically, we can extract a sequence $\{r_k\}_{k\in\mathbb{N}}$ with $r_k\to 0$ as $k\to \infty$ so that 
$$\inf\limits_{\rho\in B_r(\nu)} \relen{\rho}{\mu} = \inf\limits_{\rho\in \overline{B_r}(\nu)} \relen{\rho}{\mu}\to \relen{\nu}{\mu}.$$ 
Therefore for every $\epsilon>0$, $V$ we may find $k \in \mathbb N$ for which
\begin{equation} \label{quasi_fin}
\overline{B_{r_k}}(\nu)\subset V\quad\text{and}\quad\inf\limits_{\rho\in B_{r_k}(\nu)} \relen{\rho}{\mu} = \inf\limits_{\rho\in \overline{B_{r_k}}(\nu)} \relen{\rho}{\mu} \geq \relen{\nu}{\mu} - \frac{\epsilon}{4}\,.
\end{equation}
All in all, 
$$\att \left[ M_N(B_{r_k}(\nu)) \right]= \sum\limits_{\alpha=1}^{2^N} \prob\left[  L_{N,\alpha} \in B_{r_k}(\nu)  \right]\leq 2^N \exp\left[ -N\left(\relen{\nu}{\mu}-\frac{\epsilon}{2}\right) \right],$$
the last step by Sanov's theorem together with \eqref{quasi_fin}.  \\

\noindent The claim of the Proposition steadily follows from Markov's inequality, and $\delta=\epsilon / 4$.
\end{proof}

\begin{prop} \label{prop_2_G}
Let $\nu\in\fprosp{2^n-1}$ be such that for some $J\in\Pset{n}$, $H\left( \,  \nu^{(\Pset{J})}   \,   |  \,   \mu^{(\Pset{J})}  \, \right) > |J|\log2/n$ and $V$ an open neighborhood of $\nu$. Then there exists an open neighborhood $U$ of $\nu$, $U\subset V$, and $\delta >0$ such that for $N$ large enough
$$ \prob\left[ M_N(U) \neq 0 \right]\leq e^{-\delta N}.$$
\end{prop}

\begin{proof}
Let $J\in\Pset{n}$ satisfy $H\left( \,  \nu^{(\Pset{J})}   \,   |  \,   \mu^{(\Pset{J})}  \, \right) > |J|\log2/n$. Similarly to the proof of Proposition \ref{prop_1_G} we can select a neighborhood $W$ of $\nu^{(\Pset{J})}\in\fprosp{|\Pset{J}|}$ with 
$$\inf\limits_{\rho\in\overline{W}} H\left( \rho \mid \mu^{(\mathcal P_J)}\right) = \inf\limits_{\rho\in W} H\left( \rho \mid \mu^{(\mathcal P_J)}\right) > \frac{|J|}{n}\log2 + 2\delta,$$
for some $\delta>0$. If we now define 
$$U\ugdef \left\{ \nu \in \totprosp:\,\, \nu\in V, \nu^{(\Pset{J})} \in W \right\},$$
it follows from union bounds/subadditivity that
\beq \bea
\prob\left[ M_N(U) \neq 0 \right]& =\prob\left[ \exists \alpha: \, L_{N,\alpha} \in U \right] \leq \prob\left[ \exists \alpha: \, (L_{N,\alpha})^{(\Pset{J})} \in W \right] \\
& \leq \sum\limits_{\alpha_J=1}^{2^{|J|N/n}} \prob\left[ \, (L_{N,\alpha})^{(\Pset{J})} \in W \right]  \\
& = 2^{\frac{|J|}{n}N} \prob\left[ \, (L_{N,\alpha})^{(\Pset{J})} \in W \right]\,,
\eea \eeq
the last line using that the empirical measures are identically distributed. Since the r.h.s. above is, by Sanov's theorem, {\it at most}
\beq \bea
2^{\frac{|J|}{n}N} \exp\left[ -N\left(  \inf\limits_{\rho\in\overline{W}} H\left( \rho \mid \mu^{(\mathcal P_J)}\right) - \delta \right) \right] \leq 2^{\frac{|J|}{n}N} 2^{-\frac{|J|}{n}N} e^{-\delta N}=e^{-\delta N},
\eea \eeq
the proof of the Proposition is concluded. 
\end{proof}

\begin{prop}\label{relevant_prop}
Assume that $\nu\in\totprosp$ satisfies the $C$-conditions \eqref{def_cond} \emph{strictly}, namely that $H\left( \nu^{(\Pset{J})}|\mu^{(\Pset{J})}\right) < |J| /n\log2, \, J \in \mathcal P_n$. Let $V$ be an open neighborhood of $\nu$, and $\epsilon>0$. Then there exists an open neighborhood $U$ of $\nu$ with $U\subset V$ and $\delta>0$ such that for large enough $N$ 
$$\prob\left[ M_N(U) \leq e^{N(\log2-\relen{\nu}{\mu}-\epsilon)}\right]\leq e^{-\delta N}.$$
\end{prop}

\begin{proof}

We claim that for all $\epsilon>0$ there exists $\delta>0$ and $U\subset V$ an open neighborhood of $\nu$ for which
\begin{equation}\label{second_moment}
\var M_N(U) \leq e^{-2\delta N} \left[ \att M_N(U) \right]^2\,.
\end{equation}
Assuming this for the time being, the proposition follows steadily as an elementary application of Chebycheff's inequality: indeed, by Sanov's theorem,
\begin{equation} \label{Sanov_basic}
\att M_N(U) = 2^N\prob\left( L_{N,\alpha} \in U \right) \geq 2^N\exp\left( -N\relen{\nu}{\mu}-N\xi \,\right)\,,
\end{equation}
for $\xi>0$. Elementary manipulations then yield
\begin{equation} \label{sopra} \begin{aligned}
& \prob\left[ M_N(U) \leq 2^N\exp \left( -N\relen{\nu}{\mu} -2N\xi \right) \, \right] \\
& \quad \qquad = \prob\left[ M_N(U) - \att M_N(U) \leq e^{-N\xi } \,2^N  \exp \left( -N\relen{\nu}{\mu} -N\xi \right) - \att M_N(U)\, \right] \\
& \quad \qquad \stackrel{\eqref{Sanov_basic}}{\leq} \prob\left[ M_N(U) - \att M_N(U) \leq \left( e^{-N\xi } - 1 \right) \att M_N(U) \right] \\
& \quad \qquad \leq \prob\left[ M_N(U) - \att M_N(U) \leq -\frac{1}{2} \att M_N(U) \right], 
\end{aligned}
\end{equation}
for $N$ large enough. But the r.h.s. of \eqref{sopra} is {\it at most}
\begin{equation} \begin{aligned}
\prob\left[ | M_N(U) - \att M_N(U) | \geq \frac{1}{2} \att M_N(U) \right]  & \leq  4\frac{\var M_N(U) }{ \left[ \att M_N(U) \right]^2}\stackrel{\eqref{second_moment}}{ \leq} 4e^{-2N\delta} \leq e^{-\delta N}\,,
\end{aligned}
\end{equation}
settling the claim of the proposition. \\

\noindent It thus remains to prove \eqref{second_moment}. To see how this goes, we first observe that, under the assumptions of the proposition, there exists $\eta>0$ such that 
\begin{equation}\label{eta}
H\left(\,\nu^{(\Pset{J})}\,|\,\mu^{(\Pset{J})}\right) < \frac{|J|}{n}\log2-\eta.
\end{equation}
Moreover, we claim that for $J\in \Pset{n}$,
\begin{equation} \begin{aligned} \label{inf_3}
& \lim\limits_{r \to 0} 
\inf\limits_{ \substack{ \rho, \theta \in \overline{B}_{r}(\nu) \\ \rho^{(\Pset{J})}=\theta^{(\Pset{J})} }} \left\{ \relen{\rho}{\mu} + H\left( \,  \theta \, | \, \theta^{(\Pset{J})}\otimes \mu^{(\Pset{J}^c)}  \,\right) \right\} = \\
& \hspace{4cm} = \relen{\nu}{\mu} + H\left( \,  \nu \, | \, \nu^{(\Pset{J})}\otimes \mu^{(\Pset{J}^c)}  \,\right).
\end{aligned}\end{equation}
This can be checked through approriate upper- and lower bounds. The bound "$\leq$" is straightforward: simply plug in $\rho=\theta=\nu$ on the l.h.s. above. The bound "$\geq$" is a consequence of the semicontinuity properties of the relative entropies: in fact, by Lemma \ref{semicont}, one gets that for a sequence $(\rho_m,\theta_m)$ with $\rho_m^{(\Pset{J})}=\theta_m^{(\Pset{J})}$ and $\rho_m,\theta_m\to\nu$ it holds 
\begin{equation} \begin{aligned} 
\liminf\limits_{m\to \infty} \relen{\rho_m}{\mu} & \geq \relen{\nu}{\mu} , \\
\liminf\limits_{m\to \infty} H\left( \,  \theta_m \, | \, \theta_m^{(\Pset{J})}\otimes \mu^{(\Pset{J}^c)}  \,\right) & \geq H\left( \,  \nu \, | \, \nu^{(\Pset{J})}\otimes \mu^{(\Pset{J}^c)}  \,\right),
\end{aligned} \end{equation}
which settles the "$\geq$"-case, and thus \eqref{inf_3}. \\

\noindent By \eqref{inf_3}, for $r$ small enough such that $\overline B_r(\nu)\subset V, U=B_r(\nu)$ we have that for $J\in\Pset{n}$, it holds:
\begin{equation}  \label{bound_eta} \begin{aligned}
& \inf\limits_{ \substack{ \rho, \theta \in \overline{B}_{r}(\nu) \\ \rho^{(\Pset{J})}=\theta^{(\Pset{J})} }} \left\{ \relen{\rho}{\mu} + H\left( \,  \theta \, | \, \theta^{(\Pset{J})}\otimes \mu^{(\Pset{J}^c)}  \,\right) \right\} \\
& \hspace{3cm} \geq \relen{\nu}{\mu}  + H\left( \,  \nu \, | \, \nu^{(\Pset{J})}\otimes \mu^{(\Pset{J}^c)}  \,\right) - \frac{\eta}{2} \\
& \hspace{3cm} = \relen{\nu}{\mu} + \relen{\nu}{\mu}  - H\left( \,   \nu^{(\Pset{J})} \, | \,\mu^{(\Pset{J})}  \,\right) - \frac{\eta}{2}, 
\end{aligned} \end{equation}
the last step by the chain rule \eqref{chain_rule}. The r.h.s. above equals
\begin{equation} \begin{aligned}
& 2\relen{\nu}{\mu} - H\left( \, \nu^{(\Pset{J})} \, |  \,  \mu^{(\Pset{J})}   \, \right) - \frac{\eta}{2}  \geq  2\relen{\nu}{\mu} - \frac{|J|}{n}\log2 + \frac{\eta}{2}, 
\end{aligned} \end{equation}
the last step using \eqref{eta}. All in all, we have established that
\begin{equation} \label{eta_due}
\inf\limits_{ \substack{ \rho, \theta \in \overline U\\ \rho^{(\Pset{J})}=\theta^{(\Pset{J})} }} \left\{ \relen{\rho}{\mu} + H\left( \,  \theta \, | \, \theta^{(\Pset{J})}\otimes \mu^{(\Pset{J}^c)}  \,\right) \right\} \geq
2\relen{\nu}{\mu} - \frac{|J|}{n}\log2 + \frac{\eta}{2}\,.
\end{equation}
This inequality will play a crucial role in the control of the second moment of the $M_N$-counting randon variable which  stands behind \eqref{second_moment}: we begin by writing
\beq \bea \label{comincia}
\att M_N^2(U) & =\sum\limits_{\alpha,\alpha' = 1}^{2^N} \prob\left[ L_{\alpha,N} \in U,  L_{\alpha',N} \in U \right] \\
& \leq \sum_{ \alpha_j \neq \alpha'_j,\,\, \forall j\in I}p_{\alpha,\alpha'} 
+\,\sum_{ J\in\Pset{n}}\sum\limits_{\substack{ \alpha_j = \alpha'_j \,\, \forall \, j\in J \\ \alpha_j\neq \alpha_j' \,\, \forall j\in I\setminus J}}p_{\alpha,\alpha'}, 
\eea \eeq
where we have shortened, to lighten notations, 
\[
p_{\alpha,\alpha'} \ugdef \prob\left[ L_{\alpha,N} \in \overline{U},  L_{\alpha',N} \in \overline{U} \right].
\] 
Now, if $\alpha_j\neq\alpha'_j$ for all $j \in I$, then the associated empirical measures $L_{\alpha,N}$ and $L_{\alpha',N}$ are independent, and we may bound the corresponding contribution in \eqref{comincia} with the square of the first moment to obtain 
\begin{equation}\label{bound_second_mom}
\att M_N^2(U) \leq \left[ \att M_N(\overline{U}) \right]^2 + \,\sum\limits_{ J\in\Pset{n}}\sum\limits_{\substack{ \alpha_j = \alpha'_j \, \forall \, j\in J \\ \alpha_j\neq \alpha_j' \, \forall j\in I\setminus J}}p_{\alpha,\alpha'} .
\end{equation}

In order to control the probabilities in the sum on the r.h.s. above, we consider the generic situation where for some $J\in \Pset{n}$,  $\alpha_j=\alpha'_j \quad \forall j\in J$, $\alpha_j\neq \alpha'_j \quad \forall j\in I\setminus J$.  In this case, by the relative Sanov principle from Lemma \ref{relative_Sanov}, and for $N$ large enough, it holds:
\beq  \label{debut}
p_{\alpha, \alpha'} = \prob\left[ L_{\alpha,N} \in \overline{U},  L_{\alpha',N} \in \overline{U} \right] \leq \exp\left[ -N\inf\limits_{ \substack{\rho, \theta \in \overline{U} \\ \rho^{(\Pset{J})}=\theta^{(\Pset{J})}}}
\left\{ \mathfrak J_J(\rho, \theta) \right\}+\frac{N\eta}{4}\,\, \right]\,,
\eeq
where the quenched rate function with prescribed marginals as above satisfies  
\beq \bea
\mathfrak J_J(\rho, \theta) & = \relen{ \rho^{ \left( \Pset{J} \right) } }{ \mu^{ \left( \Pset{J} \right)} }+  H\left( \,  \rho   \,   |  \,  \rho^{\left( \Pset{J}\right)}\otimes \mu^{\left(\Pset{J}^c \right)}   \, \right)+ H\left( \,  \theta   \,   |  \,  \theta^{\left( \Pset{J}\right)}\otimes \mu^{\left(\Pset{J}^c \right)} \right) \\
& =  \relen{ \rho }{ \mu } + H\left( \,  \theta   \,   |  \,  \theta^{\left( \Pset{J}\right)}\otimes \mu^{\left(\Pset{J}^c \right)}   \, \right),
\eea \eeq
the first step by definition, and the second by the chain rule from Lemma \ref{chain_rule}. Plugging this into \eqref{debut} we thus get
\beq \bea \label{debu}
p_{\alpha, \alpha'} & \leq \exp\left[ -N\inf\limits_{ \substack{\rho, \theta \in \overline{U} \\ \rho^{(\Pset{J})}=\theta^{(\Pset{J})}}}
\left\{ \relen{ \rho }{ \mu } + H\left( \,  \theta   \,   |  \,  \theta^{\left( \Pset{J}\right)}\otimes \mu^{\left(\Pset{J}^c \right)}   \, \right) \right\}+\frac{N\eta}{4} \right] \\
& \leq 2^\frac{N|J|}{n} \exp\left[ -2N\relen{\nu}{\mu} - \frac{N\eta}{4}\right]\,,
\eea \eeq
the last inequality by the key bound \eqref{eta_due}, and therefore
\beq \bea
\sum\limits_{\substack{ \alpha_j = \alpha'_j \quad \forall j\in J \\ \alpha_j\neq \alpha_j' \quad \forall j\in I\setminus J}}p_{\alpha,\alpha'} & \leq 2^{(n-|J|)\frac{N}{n}}2^\frac{N|J|}{n}\left(2^\frac{N}{n}-1\right)^{(n-|J|)}2^\frac{N|J|}{n} \exp\left[ -2N\relen{\nu}{\mu} - \frac{N\eta}{4}\right]\\
& \leq 2^{2N} \exp\left[ -2N\relen{\nu}{\mu} - \frac{N\eta}{4}\right].
\eea \eeq
Plugging this in \eqref{bound_second_mom}, and recalling \eqref{Sanov_basic} with $\xi = \frac{\eta}{16} $, we thus get that for $N$ large enough
\beq 
\var M_N(U) \leq 2^{2N} (2^n-1)\exp\left[  -2N\relen{\nu}{\mu} -\frac{N\eta}{4} \right] \leq e^{-\frac{N\eta}{8}}\left[ \att M_N(U) \right]^2, 
\eeq
which settles the claim \eqref{second_moment} with $\delta = \frac{\eta}{16}$, say.

\end{proof}

\begin{proof}[Proof of Theorem \ref{teor_Gibbs}]
We begin with a couple of elementary observations. 
\begin{itemize}
\item First recall that for $J \in \Pset{n}$, 
\beq
C_J = \left\{ \nu \in \totprosp : \relen{ \nu^{(\Pset{J})} }{ \mu^{(\Pset{J})} } \leq \frac{|J|}{n} \log2 \right\}\,.
\eeq
These sets are compact thanks to the semicontinuity of $\nu\to\relen{\nu}{\mu}$ from Lemma \ref{semicont} . It thus follows that the intersection $\bigcap_{J\in\Pset{n}}C_J\subset \totprosp$ is itself compact, and in particular that there exists $\nu'\in\bigcap_{J\in\Pset{n}}C_J$ such that
$$\sup\limits_{\nu \in \bigcap\limits_{J\in\Pset{n}}C_J} \left\{ \int \phi d\nu-\relen{\nu}{\mu}\right\} = \int \phi d\nu'-\relen{\nu'}{\mu}.$$
\item It is well known that $\nu\to\relen{\nu}{\mu}$ is convex.
\item It clearly holds that $\mu\in\bigcap_{J\in\Pset{n}}C_J$.
\end{itemize}
Given a convex combination $\nu_\lambda\ugdef\lambda\mu+(1-\lambda)\nu'$, $0<\lambda<1$ we thus have 
$$\relen{ \nu_\lambda^{(\Pset{J})} }{ \mu^{(\Pset{J})} } < \frac{|J|}{n} \log2\qquad \forall J\in\Pset{n}.$$
We also know that as $\lambda \to 0$, $\nu_{\lambda}\to\nu'$ weakly and $\int \phi d\nu_\lambda\to\int \phi d\nu'$, $\relen{\nu_\lambda}{\mu}\to\relen{\nu'}{\mu}$. In particular, to $\epsilon>0$ we can find $\lambda>0$ with
\beq \label{conv_boh}
\int \phi d\nu_\lambda-\relen{\nu_\lambda}{\mu} \geq \int \phi d\nu'-\relen{\nu'}{\mu}-\epsilon.
\eeq
Moreover, since the map $\nu \to \int \phi(x) \nu(dx)$ is continuous, by Proposition \ref{relevant_prop} we can find a neighbourhood $U$ of $\nu_\lambda$, and $\delta>0$ such that
$$ \int \phi(x) \nu_\lambda(dx) -  \int \phi(x) \rho(dx) \leq \epsilon \quad \forall \rho\in U,$$
and with the neighbourhood $U$ such that
$$ \prob\left[ M_N(U) \leq 2^Ne^{-\relen{\nu_\lambda}{\mu}N-\epsilon N)}\right]\leq e^{-\delta N}.$$
With this choice, we thus have for the partition function of our spin glass  \eqref{partition_function} that
\beq \bea
Z_N(\phi) & = 2^{-N}\sum\limits_{\alpha=1}^{2^N} \exp\left[ N \int \phi(x) L_{N,\alpha}(dx)\right]\\
& \geq \exp \left[  N \int \phi(x) \nu_\lambda(dx) -\epsilon N \right] \exp \left[  -\relen{\nu_\lambda}{\mu}N-\epsilon N \right] \\
& \geq \exp\left[ N\sup\limits_{\nu\in\bigcap\limits_{J\in\Pset{n}}C_J} \left\{  \int \phi(x)d\nu-\relen{\nu}{\mu} \right\} -3\epsilon N    \right].
\eea \eeq
As $\epsilon$ is arbitrary, it follows from Borel-Cantelli that
\begin{equation} \label{lower_bound_final}
\liminf\limits_{N\to\infty} F_N(\phi) = \liminf\limits_{N\to\infty} \frac{1}{N}\log Z_N(\phi)  \geq \sup\limits_{\nu\in\bigcap\limits_{J\in\Pset{n}}C_J} \left\{  \int \phi(x)d\nu -\relen{\nu}{\mu} \right\}\,,
\end{equation}
almost surely. This settles the "first half of the theorem", i.e. the lower bound to the free energy.  \\

\noindent In order to establish the upper bound, we distinguish two cases:
\begin{itemize}
\item[i)]  For $\nu\in \bigcap\limits_{J\in\Pset{n}}C_J$, we choose a radius $r_\nu>0$ such that 
$$ \prob\left[ M_N(B_{r_\nu}(\nu)) \geq e^{N(\log2-\relen{\nu}{\mu}+\epsilon)}\right]\leq e^{-\delta_\nu N},$$
for large enough $N$ and some $\delta_\nu>0$; this is clearly possible in virtue of Proposition \ref{prop_1_G}.
The radius $r_\nu$ is chosen however small enough such that
$$ \left|  \int \phi(x) \rho(dx) - \int \phi(x) \nu(dx)  \right| \leq \epsilon,\qquad\forall \rho\in B_{r_\nu}(\nu),$$ 
(which is possible thanks to the aforementioned continuity of the map $\rho \mapsto \int \phi(x) \rho(dx)$).
\item[ii)] For $\nu\in C_I \setminus \left(  \bigcap\limits_{J\in\Pset{n}}C_J  \right)$ we still have $\relen{\nu}{\mu}\leq \log2$ and we can resort to Proposition \ref{prop_2_G} to find $r_\nu>0$ such that 
\begin{equation} \label{no_alphas_B}
\prob\left[ M_N(B_{r_\nu}(\nu)) \neq 0 \right] \leq e^{-\delta_\nu N},
\end{equation}
for large enough $N$, and some $\delta_\nu>0$. 
\end{itemize}

Using that $C_I \subset \bigcup_{\nu \in C_I} B_{r_\nu}(\nu)$, by compactness we may extract a finite sequence $\{\nu_j\}_{j=1}^m\subset C_I$ for which covering still holds, to wit
$$C_I \subset \bigcup_{j=1}^m B_{r_j}(\nu_j),$$
where we have also shortened $r_j\ugdef r_{\nu_j}$. Moreover, if we set
$$ U \ugdef  \bigcup_{j=1}^m B_{r_j}(\nu_j),\qquad \delta\ugdef \min\limits_{j\in\{1,\cdots,m\}} \delta_{\nu_j},$$
it follows from Sanov theorem that
\beq \label{sanov_good}
\limsup\limits_{N\to\infty} \frac{1}{N} \log \prob\left[  L_{N,\alpha}\notin U  \right]\leq -\inf\limits_{\nu\notin U} \relen{\nu}{\mu} < -\log2.
\eeq
Splitting the contributions to the partition function then steadily yields the upper bound
\beq \bea 
Z_N(\phi) & \leq 2^{-N} \sum_{l=1}^m \, \sum\limits_{\alpha:  L_{N,\alpha}\in B_{r_l}(\nu_l)} \, \exp{\left[ N \int \phi(x)L_{N,\alpha}(dx) \right]} + \\
& \hspace{5cm} + 2^{-N} \sum\limits_{\alpha:  L_{N,\alpha}\notin U} \exp{\left[\int \phi(x)L_{N,\alpha}(dx) \right]}\,.
\eea \eeq
By \eqref{sanov_good},  the second summand yields no contribution in the large $N$-limit. Analogously, \eqref{no_alphas_B} establishes that the same applies to those terms in the first summand for which $\nu_l\notin  \bigcap\limits_{J\in\Pset{n}}C_J $. All in all we have 
\beq \bea
Z_N & \leq 2^{-N} \sum\limits_{l:\,\, \nu_l\in\bigcap\limits_{J\in\Pset{n}}C_J}  \,\, \sum\limits_{\alpha:  L_{N,\alpha}\in B_{r_l}(\nu_l)} \, \exp{\left[ N \int \phi(x)L_{N,\alpha}(dx) \right]} \\
& \leq e^{\epsilon N}  \sum\limits_{l:\,\, \nu_l\in\bigcap\limits_{J\in\Pset{n}}C_J} \exp{\left[ N \int \phi d \nu_l \right]} M_N\left( B_{r_l}(\nu_l)  \right) \\
& \leq e^{2\epsilon N}   \sum\limits_{l:\,\, \nu_l\in\bigcap\limits_{J\in\Pset{n}}C_J} \exp{\left[ N \int \phi d \nu_l -N\relen{\nu_l}{\mu}\right]} \\
& \leq  e^{2\epsilon N} m \exp{\left[  N\sup\limits_{\nu\in  \bigcap\limits_{J\in\Pset{n}}C_J} \left\{ \int \phi d \nu -\relen{\nu}{\mu}  \right\}\right]}
\eea \eeq
almost surely, for large enough $N$.  Since $\epsilon$ is arbitrary we thus have
$$\limsup\limits_{N\to\infty} F_N(\phi) = \limsup\limits_{N\to\infty} \frac{1}{N} \log Z_N(\phi) \leq   \sup\limits_{\nu\in\bigcap\limits_{J\in\Pset{n}}C_J} \left\{  \int \phi d \nu -\relen{\nu}{\mu} \right\},$$
almost surely. This being the sought upperbound, Theorem \ref{teor_Gibbs} follows.
\end{proof}

\section{The Parisi variational principle}
In this section we prove Theorem \ref{solution_T_Par}.  The following upper bound 
\begin{equation}\label{upper_bound}
f(\phi) \leq \min\limits_{T\in\chset}f^T(\phi),
\end{equation} 
is immediate: since for any $T\in\chset$,
\[
\bigcap\limits_{J\in\Pset{n}}C_J\subset\bigcap\limits_{k=1}^n C_{A^T_k}\,,
\]
the Boltzmann-Gibbs variational principle \eqref{quenched_limit} is analogous to the ones in \eqref{free_en_T} with the only difference that the supremum in the former  is taken on a {\it smaller} set. This proves that $f(\phi)\leq f^T(\phi)$ for any $T\in\chset$, and thus settles \eqref{upper_bound}. 

On the other hand, the Parisi variational principle \eqref{Parisi_T} established in \cite{bk2}, states that for the  Perceptron GREM  
\beq \label{P_T}
f^T(\phi) = \min_{\boldsymbol m \in \Delta} P^T(\boldsymbol m).
\eeq
Combining \eqref{upper_bound} and \eqref{P_T} we thus have
\beq
f(\phi) \leq \min_{T \in \mathcal C_n} \left\{ \min_{m \in \Delta} P^T(\boldsymbol m) \right\}\,,
\eeq
which is "half of the theorem". 

In order to prove the "second half", namely the lower bound, some infrastructure is needed. Following \cite{bk2}, for each $T\in\chset$ we define a probability distribution $G^T = G^T(\boldsymbol m)$ on $S^{2^n-1}$ which depends on the parameter $\boldsymbol{m}\in\Delta$. This {\it generalized Gibbs measure} is associated to the Parisi functional $P^T$, recall \eqref{func_Par} above, and is described through
\begin{itemize}
\item a measure $\gamma^T$ on $S$,  
\item and, for $2\leq j \leq n$, Markov kernels $K^T_j$ with source $S^{|\Pset{j-1}^T|}=S^{2^{j-1}-1}$ and target $S^{|T_{j}|}=S^{2^{j-1}}$.
\end{itemize}
Precisely, and recalling the functions $\left\{ \phi^T_k \right\}_{k=1}^n$ from \eqref{def_phik}, the generalized Gibbs measure $\gt$ is given by the semi-direct product
\begin{equation}\label{G_meas}
\gt = \gamma^T\otimes K^T_2\otimes \cdots \otimes K^T_n.
\end{equation}
$$ \gamma^T(dx) \ugdef \frac{\exp\left[m_1 \phi_1^T(x)\right]\mu_{A^T_1}(dx)}{\exp\left[m_1 \phi_0^T(x)\right]},$$
$$ K^T_j( \vx{1},\cdots,\vx{j-1},d\vx{j}) \ugdef  \frac{\exp\left[m_j \phi_j^T(\vx{1},\cdots,\vx{j})\right]\mu^{(T_{j})}(d\vx{j})}{\exp\left[m_j \phi_{j-1}^T(\vx{1},\cdots,\vx{j-1})\right]}$$
where\footnote{Notice that with this notation it holds in particular that $(\vx{1},\cdots,\vx{j-1})=(x_J)_{J\in\Pset{j-1}^T}\in S^{2^{j-1}-1}.$} $\vx{j}=\left(x_J\right)_{J\in T_{j}}\in S^{2^{j-1}}$ We also write 
\begin{equation}\label{G_marginal_P}
 \gt_{j-1} \ugdef \pgt^{ ( \Pset{ j-1 }^T ) } \equiv \gamma^T\otimes K^T_2\otimes \cdots \otimes K^T_{j-1}
\end{equation}
for the marginal of $\gt$ on coordinates corresponding to elements of $\Pset{j-1}^T$.

In virtue of \cite{bk2}, the  generalized Gibbs measure $G^T$ is the unique solution to the Boltzmann-Gibbs variational principle to given chain: precisely, with 
\beq
\boldsymbol m \ugdef \arg\min_{\boldsymbol n \in \Delta} P^T( \boldsymbol n)
\eeq
the optimal parameter in Parisi variational principle associated to a $T$-chain, and recalling the Boltzmann-Gibbs variational principle \eqref{free_en_T}, it follows  from \cite[Theorem 2.3-2.5]{bk2} that
\beq
G^T(\boldsymbol m) = \arg\sup_{\nu\in\totprosp} \left\{ \Phi(\nu) - \mathfrak J^T(\nu) \right\}\,.
\eeq
In particular, the measure $G^{T}(\boldsymbol m)$ satisfies all constraints for the empirical measures in the associated Boltzmann-Gibbs principle for the chain $T$, i.e.
\beq \label{sidecon}
G^{T}(\boldsymbol m) \in \bigcap\limits_{k=1}^n C_{A_k^T}.
\eeq
Without loss of generality\footnote{Such form of the minimal chain can always be achieved by re-labelling.}, we shall assume here that the {\it minimal chain} 
\beq 
\mathcal T \ugdef \arg\min\limits_{T \in\chset} f^{T}(\phi)\,,
\eeq
i.e. the one which minimizes {\it all} Parisi variational principles, is given by  
\beq  \label{T_min}
\mathcal T =\left\{A^{\mathcal T}_k\right\}_{k=0}^n\quad\text{with}\quad A^{\mathcal T}_k\ugdef\{1,2,\cdots ,k-1,k\}.
\eeq
We now claim that the generalized Gibbs measure $G^{\mathcal T}_{\boldsymbol m'}$ associated to the minimal chain satisfies {\textit{all}} constraints \eqref{def_cond} or, which is the same,  that
\beq \label{Gt_sat}
\boldsymbol m' \ugdef \arg \min_{\boldsymbol m \in \Delta} P^{\mathcal T}(\boldsymbol{m}) 
\Longrightarrow  G^{\mathcal T}_{\boldsymbol{m}'} \in \bigcap_{J\in\Pset{n}} C_J.
\eeq
This amounts to saying that the measure $G^{\mathcal T}_{\boldsymbol m'}$ is a viable candidate for the Boltzmann-Gibbs variational principle for the nonhierarchical Perceptron GREM: Theorem \ref{teor_Gibbs} would then imply that 
\beq \bea 
f(\phi) & \geq \int \phi d G^{\mathcal T}(\boldsymbol m') - H\left( G^{\mathcal T}(\boldsymbol m') \mid \mu \right) \\
& = \min_{T \in \mathcal C_n} \left\{ \min_{m \in \Delta} P_T(\boldsymbol m) \right\},
\eea \eeq
with the second line by definition/construction of $\mathcal T, \boldsymbol m'$, and Theorem \ref{solution_T_Par} would follow.\\

\noindent It therefore remains to prove the key claim \eqref{Gt_sat}. This requires some preliminary observations. 
Recalling that $\phi^T_{l}$ depends on $m_j$ only for $l\in\{1,\cdots,j-1\}$ a simple computation yields 
\beq \bea 
\partial_j P^T(\boldsymbol{m}) = \frac{1}{m_j^2} \left[
\int H\left( \, K_j^T(\vx{1},\cdots,\vx{j-1},\cdot)\, | \, \mu^{\left( \Bj \right)} \, \right) \, G^T_{j-1}(d\vx{1},\cdots,d\vx{j-1})
-\frac{\log2}{n}\right], 
\eea \eeq
cfr. also \cite[eq. 3.13]{bk2}.  Using that $G^T_{j-1}\otimes K^T_j \equiv G^T_{j}$, the integral in the above can be written as
\beq \bea
& \int H\left( \, K_j^T(\vx{1},\cdots,\vx{j-1},\cdot)\, | \, \mu^{\left( \Bj \right)} \, \right) \, G^T_{j-1}(d\vx{1},\cdots,d\vx{j-1})  \\
& \hspace{3cm} =  \int \log\left(f(\vx{j}|\vx{1},\cdots,\vx{j-1})\right) G^T_j(d\vx{1},\cdots,d\vx{j})
\eea \eeq
where we shortened $f(\cdot|\vx{1},\cdots,\vx{j-1})$ for the Radon-Nikodym derivative of the measure $K_j^T(\vx{1},\cdots,\vx{j-1},\cdot)$ with respect to $\mu^{\left( T_j \right)}$. Following the steps \eqref{comi_uno}-\eqref{comi_due}
with $\mu^{(B^c)}:= \mu^{\left( \Bj \right)}$, $\nu^{(B^c|B)} := K^T_j$ (in which case  $\nu^{(B)}:= G^T_{j-1}$, $\nu :=  G^T_j$) we get 
$$ \int \log\left(f(\vx{j}|\vx{1},\cdots,\vx{j-1})\right) G^T_j(d\vx{1},\cdots,d\vx{j}) = H\left( \, G^T_j \, | \,  G^T_{j-1}\otimes\mu^{(\Bj)}\right),$$
which implies
\begin{equation}\label{derivative}
\partial_j P^T(\boldsymbol{m}) =\frac{1}{m_j^2}\left[ H\left( \, G^T_j \, | \,  G^T_{j-1}\otimes\mu^{(\Bj)}\right) - \frac{\log2}{n} \right].
\end{equation}
Coming back to our main task, a proof that the minimal chain satisfies all side constraints in the Boltzmann-Gibbs principle, let 
\beq \label{choice}
J=\left\{j_1<j_2<\cdots<j_{M-1}<j_M\right\}\in\Pset{n}
\eeq
be a generic non-empty subset of $I$. The goal is thus to show  that
\begin{equation}\label{cond_Gt_J}
H\left(\, \left( G^{\mathcal T}\right)^{(\Pset{J})} \, | \, \mu^{(\Pset{J})} \, \right) \leq \frac{M}{n}\log2.
\end{equation}
Recall the definition \eqref{index_k} of the index $k_J$ with the property that $J \in \mathcal{T}_{k_J}$:
under \eqref{T_min} it holds that ${\mathcal T}_k = \Pset{k}^{\mathcal T} \setminus \Pset{k-1}^{\mathcal T} \equiv \Pset{(1,\cdots,k)}\setminus \Pset{(1,\cdots, k-1)}, k\in I$, hence
$$k_J=\max\{j:\, j\in J\}=j_M\,\geq\, M.$$
In light of \eqref{derivative}, which shows how the conditions on the relative entropies relate to the partial derivatives of the Parisi functions, it is natural to distinguish two cases, depending on the latter being  negative or positive. \\

\noindent {\bf First case: $\partial_{k_J} P^{\mathcal T}(\boldsymbol{m}')\leq 0$.}  It then follows from \eqref{derivative} 
\begin{equation} \label{deriv_non_pos}
H \left( G^{\mathcal T}_{k_J} \,  | \, G^{\mathcal T}_{k_J-1}\otimes\mu^{( {\mathcal T}_{k_J}   )} \right) 
\leq \frac{\log2}{n}.
\end{equation}
We distinguish two sub-cases:
\begin{itemize}
\item[i)] When $M=1$, namely when $J=\{ j_1\}$ and $k_J = j_1$, we have $\Pset{J} = \{J\} \subseteq {\mathcal T}_{k_J}$. In particular, the marginal of $\,\,G^{\mathcal T}_{k_J-1}\otimes \mu^{({\mathcal T}_{k_J})}\,\,$ on $\mathcal P_J$ is $\mu^{(\mathcal P_J)}=\mu_{J}$ hence 
\beq \bea \label{uno}
H\left( (G^{\mathcal T})^{(\Pset{J})} \, | \, \mu^{(\Pset{J})} \right) & = H\left( (G^{\mathcal T})^{(\{J\})} \, | \, \mu_{J} \right) \\
& \stackrel{\eqref{AM}}{\leq} H \left( G_{k_J}^{\mathcal T} \,  | \, G^{\mathcal T}_{k_J-1}\otimes\mu^{({\mathcal T}_{k_J})} \right) \stackrel{\eqref{deriv_non_pos}}{\leq} \frac{\log2}{n},
\eea \eeq
which establishes \eqref{cond_Gt_J} for $J$-s of cardinality one.\\

\item[ii)] For sets of cardinality $M=2,\dots,n$ we proceed inductively: with 
$$
{\mathfrak O} \ugdef J\setminus{\{k_J\}}=\{j_1,j_2,\dots,j_{M-1}\}\,,
$$
we assume  that
\beq \label{indu}
H\left(   (G^{\mathcal T})^{(\Pset{{\mathfrak O}})} \,|\,  \mu^{(\Pset{{\mathfrak O}})}   \right) \leq \frac{M-1}{n}\log2\,.
\eeq
and claim that
\beq \label{aswell}
H\left( (G^{\mathcal T})^{(\Pset{J})} \, | \, \mu^{(\Pset{J})} \right) \leq \frac{M}{n} \log 2. 
\eeq
To see this we first observe that, evidently, $ |\mathfrak O| = M-1$, and 
$$\Pset{\mathfrak O}\subseteq \Pset{k_J-1}^{\mathcal T},\quad\Pset{J}\subseteq \Pset{{\mathfrak O}}\cup \mathcal{T}_{k_J}, \qquad \Pset{{\mathfrak O}}\cap \mathcal{T}_{k_J} = \emptyset,$$
hence by the monotonicity property \eqref{AM} it holds 
\beq \bea \label{baa}
& H\left( (G^{\mathcal T})^{(\Pset{J})} \, | \, \mu^{(\Pset{J})} \right) \leq  H\left( (G^{\mathcal T})^{(\Pset{{\mathfrak O}}\,\cup\, \mathcal{T}_{k_J})} \, | \, \mu^{(\Pset{{\mathfrak O}}\,\cup \,\mathcal{T}_{k_J})} \right) \\
& \stackrel{\eqref{chain_rule}}{=} H\left(   (G^{\mathcal T})^{(\Pset{{\mathfrak O}})} \,|\,  \mu^{(\Pset{{\mathfrak O}})}   \right) +  H\left(  (G^{\mathcal T})^{(\Pset{\mathfrak O}\,\cup\, \mathcal{T}_{k_J})} \, | \,  (G^{\mathcal T})^{(\Pset{{\mathfrak O}})}\otimes\mu^{(\mathcal{T}_{k_J})}  \right) \\
& \stackrel{\eqref{indu}}{\leq} \frac{M-1}{n} \log 2 + H\left(  (G^{\mathcal T})^{(\Pset{{\mathfrak O}}\,\cup\, \mathcal{T}_{k_J})} \, | \,  (G^{\mathcal T})^{(\Pset{{\mathfrak O}})}\otimes\mu^{(\mathcal{T}_{k_J})}  \right)\,.
\eea \eeq
For the second summand above we write
\beq \bea \label{baaa}
H\left(   (G^{\mathcal T})^{(\Pset{{\mathfrak O}}\,\cup\,\mathcal T_{k_J})} \, | \,  (G^{\mathcal T})^{(\Pset{\mathfrak O})}\otimes\mu^{(\mathcal T_{k_J})}  \right) & \stackrel{\eqref{AM}}{\leq} H \left( G^{\mathcal T}_{k_J} \,  | \,  G^{\mathcal T}_{k_J-1} \otimes \mu^{(\mathcal{T}_{k_J})} \right) \\
& \stackrel{\eqref{deriv_non_pos}}{\leq} \frac{\log2}{n},
\eea \eeq
Plugging \eqref{baaa} into \eqref{baa} therefore yields
\beq
H\left( (G^{\mathcal T})^{(\Pset{J})} \, | \, \mu^{(\Pset{J})} \right) \leq \frac{M-1}{n}\log2 + \frac{\log2}{n} = \frac{M}{n}\log2, 
\eeq
settling the induction step \eqref{aswell}. \\
\end{itemize}

\noindent{\bf Second case: $\partial_{k_J} P^{\mathcal T}(\boldsymbol{m}')>0$.} It will become clear in the treatment that this is a somewhat degenerate situation where the minimal chain is not unique. As a matter of fact, we will use such lack of uniqueness to our advantage, in sofar we will deduce the validity of the side constraints from the properties of the Parisi functional evaluated in yet a second, well chosen (and also: minimal) chain. To see how this goes, we first claim that for $k\in\{2,\cdots,n\}$ the following implication holds true:
\beq \label{second_case}
\partial_{k} P^{\mathcal T}(\boldsymbol{m}')>0  \Longrightarrow \; \forall\,j=1,\cdots,k-1: \, m'_j = m'_{j+1}\,.
\eeq
We will prove the following, fully equivalent implication: with
$$ l = l(\mathcal T,k) \ugdef \min\left\{\,\, j : j\in\{1,2,\cdots,k-2,k-1\};\,\,\, m'_j = m'_{j+1} \,\, \right\},$$
we claim that 
\beq  \label{impli_basta} 
\partial_{k} P^{\mathcal T}(\boldsymbol{m}')>0  \Longrightarrow l = 1. 
\eeq
Proceeding by contradiction, let us assume that $l > 1$, in which case $l-1\geq 1$ and 
$$
m'_{l-1} < m'_l = m'_{l+1} = \cdots = m'_{k-1} =  m'_{k}  \leq m'_{k+1} 
$$
(with $m'_{n+1}=1$). Since $\boldsymbol{m}'= \arg \min_{\boldsymbol{m}\in\Delta} P^{\mathcal T}(\boldsymbol{m})$, the function 
$$m_k\in[ m'_{l-1} , m'_{k+1}] \, \rightarrow P^{\mathcal T}(m'_1, \cdots, m'_{k-1},m_k,m'_{k+1},\cdots,m'_n)$$
assumes its minimum in $m_k=m'_k\in( m'_{l-1} , m'_{k+1}]$ and this implies $\partial_{k} P^{\mathcal T}(\boldsymbol{m}')\leq 0$,  which contradicts our working assumption {\it unless} $l=1$, hence
\beq \label{tutti_uguali}
m_1' = m_2' = \dots = m_{k}',
\eeq
and \eqref{impli_basta} is settled.  

The next, and final, step is to show - as anticipated - that if $l(\mathcal T, k_J)=1$ we may find yet another chain $\mathcal R$ with the same "Parisi free energy" as the chain $\mathcal T$, but from which we may also deduce the validity of the side constraints \eqref{cond_Gt_J}, namely that for $J = \{ j_1 < j_2 < \cdots < j_M\}$, 
\beq
H\left(\, \left( G^{\mathcal T}\right)^{(\Pset{J})} \, | \, \mu^{(\Pset{J})} \, \right) \leq \frac{M}{n}\log2.
\eeq
Towards this goal, define $$\{b_1<\cdots<b_{j_M-M}\}\ugdef A^{\mathcal T}_{j_M}\setminus J\equiv \{1,2,\cdots,j_M-1,j_M\}\setminus \{j_1,\cdots, j_M\},$$ denote by $\mathcal R=\left\{A^{\mathcal R}_k\right\}_{k=0}^n\in\chset$ the chain with 
\begin{equation}\label{T_tilde}
A^{\mathcal R}_k \ugdef \begin{cases}
(j_1,j_2,\cdots,j_{k-1},j_k) & \text{if}\quad k\in\{1,\cdots,M\}\\
J\cup\left\{b_1,\cdots,b_{k-M}\right\} & \text{if}\quad k\in\{M+1,\cdots,j_M-1\}\\
A^{\mathcal T}_k & \text{if}\quad k\in\{j_M,\cdots,n\} \end{cases},
\end{equation}
and consider the corresponding extremal parameter $$\boldsymbol{m}''=\text{argmin}_{\boldsymbol{m}\in\Delta}P^{\mathcal R}(\boldsymbol{m}).$$ 
In light of \eqref{tutti_uguali}, with $k= k_J$ and since $k_J = j_M$ we have for the coordinates of the extremal parameter $\boldsymbol{m}'$ associated to $\mathcal T$ that $m_1' = m_2' = \dots = m_{j_M}'$ and therefore 
\[
\frac{m'_l}{m'_{l+1}} = 1, \qquad l=1,\cdots,j_M-1.
\] 
Due to the recursive construction \eqref{def_phik}, and shortening $\vx{j}=(x_J)_{J\in\mathcal T_j}$,  it follows that
\beq \bea  \label{phizero_T}
\phi^{\mathcal T}_0(\boldsymbol m') &=\frac{1}{m_1'}\log\int e^{{m_1'}\phi^{\mathcal T}_1(\vx{1})}\mu^{(\mathcal{T}_1)}(d\vx{1}) \\
& =\frac{1}{m_1'}\log \int \exp \left[\frac{m'_1}{m'_2}\log \int e^{m_2'\phi^{\mathcal T}_2(\vx{1},\vx{2})}\mu^{(\mathcal{T}_2)}(d\vx{2})\right]\mu^{({\mathcal T}_1)}(d\vx{1}) \\
& = \frac{1}{m_1'}\log\int \exp\left[{m_2' \phi^{\mathcal T}_2(\vx{1},\vx{2})}\right] \mu^{({\mathcal{T}}_1)} \otimes\mu^{({\mathcal T}_2)}(d\vx{1},d\vx{2})\,.
\eea \eeq
Again since $m_1' = m_2'$, the r.h.s. of \eqref{phizero_T} equals
\beq \bea
& \frac{1}{m_1'}\log\int \exp\left[{m_1' \phi^{\mathcal T}_2(\vx{1},\vx{2})}\right] \mu^{({\mathcal{T}}_1)} \otimes\mu^{({\mathcal T}_2)}(d\vx{1},d\vx{2}) \\
& \qquad =\frac{1}{m_1'}\log\int \exp\left[{m_1' \phi^{\mathcal T}_{j_M}(\vx{1},\vx{2},\cdots,\vx{j_M})}\right] \bigotimes_{l=1}^{j_M}\mu^{({\mathcal{T}}_l)} (d\vx{l}),
\eea \eeq
the last equality by iteration. All in all,  
\beq \label{heila}
\phi^{\mathcal T}_0(\boldsymbol m') = \frac{1}{m_1'}\log\int \exp\left[{m_1' \phi^{\mathcal T}_{j_M}(\vx{1},\vx{2},\cdots,\vx{j_M})}\right] \bigotimes_{l=1}^{j_M}\mu^{({\mathcal{T}}_l)} (d\vx{l}), 
\eeq
and in complete analogy for the $\mathcal R$-chain but this time with $\vx{j} \ugdef (x_J)_{J\in{\mathcal R}_j}$
\beq \label{heila_two}
\phi^{\mathcal R}_0(\boldsymbol m')= \frac{1}{m_1'}\log\int \exp\left[{m_1' \phi^{\mathcal R}_{j_M}(\vx{1},\vx{2},\cdots,\vx{j_M})}\right] \bigotimes_{l=1}^{j_M}\mu^{({\mathcal{R}}_l)} (d\vx{l}).
\eeq
Since $\bigcup_{l=1}^{j_M} \mathcal T_l = \mathcal P^{\mathcal T }_{j_M}\equiv \mathcal P^{\mathcal R }_{j_M}=\bigcup_{l=1}^{j_M}  \mathcal R_l$, the product measures in \eqref{heila} and \eqref{heila_two} are the same, to wit 
\beq \label{int_same}
\bigotimes_{l=1}^{j_M}\mu^{({\mathcal{R}}_l)}=\bigotimes_{l=1}^{j_M}\mu^{({\mathcal{T}}_l)}=\mu^{(\mathcal{P}^{\mathcal T}_{j_M})}=\bigotimes_{J\in \mathcal{P}_{\{1,\cdots, j_M\}}} \mu_J(dx_J)\,.
\eeq
Furthermore, the form taken by the functions $\phi^{\mathcal R}_{j_M}(\boldsymbol m'),\cdots,\phi^{\mathcal R}_{n}(\boldsymbol m')$ and respectively $\phi^{\mathcal T}_{j_M}(\boldsymbol m'),\cdots,\phi^{\mathcal T}_{n}(\boldsymbol m')$ depends on the sets $\{A_k^{\mathcal T}\}_{k=J_M+1}^n$ only, respectively $\{A_k^{\mathcal R}\}_{k= J_M+1}^n$: as these sets coincide, it follows, in particular,  that 
\beq \label{quasi_fine}
\phi^{\mathcal R}_{j_M}(\boldsymbol m')\equiv\phi^{\mathcal T}_{j_M}(\boldsymbol m').
\eeq 
Combining \eqref{heila}-\eqref{quasi_fine}, we therefore see that 
\begin{equation} \bea  \label{bo}
P^{\mathcal T}(\boldsymbol{m}') & = \frac{\log 2}{n} \sum_{k=1}^n \frac{1}{m_k'} + \phi_0^{\mathcal T}(\boldsymbol m')  = \frac{\log 2}{n} \sum_{k=1}^n \frac{1}{m_k'} + \phi_0^{\mathcal R}(\boldsymbol m')  = P^{\mathcal R}(\boldsymbol{m}')\,,
\eea 
\end{equation}
and by similar reasoning, 
\beq \label{bo1}
G^{{\mathcal T}}(\boldsymbol{m}') =G^{\mathcal R}(\boldsymbol{m}').
\end{equation}
From \eqref{bo} and minimality of the chain $\mathcal T$, i.e.
$$
P^{\mathcal T}(\boldsymbol{m}')=\min\limits_{\boldsymbol{m}\in\Delta}P^{\mathcal T}(\boldsymbol{m}) \leq \min\limits_{\boldsymbol{m}\in\Delta}P^{\mathcal R}(\boldsymbol{m}) = P^{\mathcal R}(\boldsymbol{m}''),
$$
if follows in particular that 
\beq
P^{\mathcal R}(\boldsymbol m') = P^{\mathcal R}(\boldsymbol m''). 
\eeq
By uniqueness of the minimum of the Parisi principle, see \cite[Prop. 3.6]{bk2}, it thus follows that $\boldsymbol m' = \boldsymbol m''$, which together with \eqref{bo1} implies
\beq
G^{{\mathcal T}}(\boldsymbol{m}') =G^{\mathcal R}(\boldsymbol{m}') = G^{\mathcal R}(\boldsymbol{m}'')\,.
\eeq
Since $A^{\mathcal R}_{M}=J$, we  have
\beq \bea
H\left( \, (G^{\mathcal T})^{(\Pset{J})} \, | \, \mu^{(\Pset{J})} \, \right) &= H\left( \, [G^{\mathcal T}(\boldsymbol{m}')]^{(\Pset{J})} \, | \, \mu^{(\Pset{J})} \, \right) \\
& =  H\left( \, [G^{\mathcal R}(\boldsymbol{m}'')]^{(\Pset{M}^{\mathcal R})} \, | \, \mu^{(\Pset{M}^{\mathcal R})} \, \right) 
\leq \frac{M}{n}\log2,
\eea \eeq
the last inequality since $G^{\mathcal R}(\boldsymbol{m}'')\in C_{A^{\mathcal R}_{M}}$, i.e. the measure $G^{\mathcal R}(\boldsymbol{m}'')$ satisfies the side constraints on the empirical measures. This concludes the proof of Theorem \ref{solution_T_Par}. \\

\hfill $\square$

\end{document}